\documentclass[11pt,reqno]{amsart}
\usepackage{amsmath, amssymb, amsthm}
\usepackage{url}
\usepackage[breaklinks]{hyperref}
\renewcommand{\(}{\left\(}
\renewcommand{\)}{\right\)}
\renewcommand{\[}{\left\[}
\renewcommand{\]}{\right\]}

\numberwithin{equation}{section}
 \theoremstyle{plain}
\newtheorem{theorem}{Theorem}[section]

\newtheorem{corollary}[theorem]{Corollary}
\newtheorem{proposition}[theorem]{Proposition}

   \makeatletter
\def\proof{\@ifnextchar[{\@oproof}{\@nproof}}
\def\@oproof[#1][#2]{\trivlist\item[\hskip\labelsep\textit{#2 Proof of\
#1.}~]\ignorespaces}
\def\@nproof{\trivlist\item[\hskip\labelsep\textit{Proof.}~]\ignorespaces}
\makeatother

\begin{document}
\title[Minimal excludant over partitions into distinct parts ]{Minimal excludant over partitions into distinct parts}

%\author{Subhash Chand Bhoria, Pramod Eyyunni, Prabhsimrat Kaur and Bibekananda Maji}

\author{Prabh Simrat Kaur}
\address{Prabh Simrat Kaur,  School of Mathematics\\
Thapar Institute of Engineering and Technology\\
Patiala 147004,  India.}
\email{prabh.simrat17@gmail.com}

\author{Subhash Chand Bhoria}
\address{Subhash Chand Bhoria, Pt. Chiranji Lal Sharma Government College, Urban Estate, Sector-14, Karnal,  Haryana 132001, India.}
\email{scbhoria89@gmail.com}

\author{Pramod Eyyunni}
\address{Pramod Eyyunni, Discipline of Mathematics,
Indian Institute of Technology Indore, Simrol, Indore, Madhya Pradesh - 453552, India. \newline
(Former) Indian Institute of Science Education and Research Berhampur, Industrial Training Institute (ITI) Berhampur, Engineering School Road, Berhampur, Odisha - 760010, India.} 
\email{pramodeyy@gmail.com}

\author{Bibekananda Maji}
\address{Bibekananda Maji\\ Discipline of Mathematics \\
Indian Institute of Technology Indore \\
Indore, Simrol, Madhya Pradesh 453552, India.} 
\email{bibekanandamaji@iiti.ac.in}

\thanks{$2020$ \textit{Mathematics Subject Classification.} Primary 11P81, 11P82,  05A17 ; Secondary 05A19. }

%\address{Subhash's address, \\ Indian Institute of Science Education and Research Berhampur, Industrial Training Institute (ITI)
% Berhampur, Engineering School Road, Berhampur, Odisha - 760010, Simrat's address, \\ Discipline of Mathematics
% Indian Institute of Technology Indore, Simrol, Indore - 453552.}
%\email{scbhoria89@gmail.com, pramodeyy@gmail.com, prabh.simrat17@gmail.com, bibekanandamaji@iiti.ac.in}

\begin{abstract}
The average size of the ``smallest gap" of a partition was studied by Grabner and Knopfmacher in 2006. Recently, Andrews and Newman, motivated by the work of Fraenkel and Peled, studied the concept of the ``smallest gap" 
under the name ``minimal excludant" of a partition and rediscovered a  result of Grabner and Knopfmacher. In the present paper, we study the 
sum of the minimal excludants over partitions into distinct parts,
and interestingly we observe that it has a nice connection with Ramanujan's function $\sigma(q)$. As an application, we derive a stronger version of a result of Uncu.
\end{abstract}
\maketitle

\section{Introduction}\label{intro}
Grabner and Knopfmacher \cite{grabner} studied an interesting partition statistic under the name `smallest gap'. They defined the smallest gap of an integer partition as the least integer missing from the partition.  Fraenkel and Peled introduced the concept of a {\it minimal excludant} of a set $S$ of positive integers, namely, the least positive integer missing from the set, denoted by ``$\textup{mex}(S)$". Recently, in $2019$, Andrews and Newman explored the idea of minimal
excludant and in the process, rediscovered a result of Grabner and Knopfmacher on “smallest
gap”. They very naturally generalized this concept to other arithmetic progressions in their
two papers \cite{andrewsnewmanI,  andrewsnewmanII}.  Let us define
\begin{equation}\label{sigmamex}
 \sigma\textup{mex}(n) := \sum_{\pi \in \mathcal{P}(n)} \textup{mex}(\pi),
\end{equation}
where $\mathcal{P}(n)$ denotes the collection of all integer partitions of $n$. Interestingly,
Andrews and Newman \cite[Theorem 1.1]{andrewsnewmanI} proved that
\begin{equation}\label{combinatorial_sigmamex}
\sigma\textup{mex}(n) = D_2(n),
\end{equation}
where $D_2(n)$ represents the number of two-colored partitions of $n$ into distinct parts. 
The generating function of the above identity was, in fact, obtained by Grabner and Knopfmacher \cite[Theorem 3]{grabner} in their work on the 
smallest gap. They also obtained the following Hardy-Ramanujan-Rademacher type exact formula for  $ \sigma\textup{mex}(n)$:
\begin{align}\label{HRR type}
\sigma\textup{mex}(n) = \frac{\pi}{ 2 \sqrt{6\left(n+ \frac{1}{12}  \right)} } \sum_{k=1}^{\infty} \frac{A_{2k-1}(n)}{2k-1} I_1\left( \pi  \frac{ \sqrt{ 2\left(n+ \frac{1}{12}  \right)} }{\sqrt{3} (2k-1)} \right),
\end{align}
where 
\begin{align*}
A_{k}(n) = \sum_{0 \leq h <  k \atop \gcd(h,k)=1} \exp \left( 2 \pi i \left( s(h,k) - s(2h,k) - \frac{hn}{k}  \right)   \right),
\end{align*}
and $s(h,k)$ denotes the Dedekind sum,  and $I_1$ denotes the modified Bessel function of the first kind with index $1$.
Moreover,  $\sigma\textup{mex}(n)$ satisfies the following asymptotic formula:
\begin{align*}
\sigma\textup{mex}(n) \sim  \frac{1}{4  \sqrt[4]{6n^3}} \exp\left( \pi \sqrt{\frac{2n}{3}}   \right) \quad \textrm{as} \, \, n \rightarrow \infty.
\end{align*}
%with the help of generating function.
%Sellers and da Silva in \cite{sellers-dasilva} prove some results on $\textup{mex}_{a, A}(\pi)$ for certain values of the parameters $a$ and $A$.
%WRITE THE RESULTS. \\
A combinatorial proof of \eqref{combinatorial_sigmamex} has been obtained by Ballantine and Merca \cite{ballantine-merca}.
Recently, Chern \cite{chernmaxexcludnt} defined {\it maximal excludant} $``\textup{maex}(\pi)"$ as the largest non-negative integer smaller than the largest part of $\pi$ that is not a part of $\pi$. Analogous to \eqref{sigmamex}, Chern defined
\begin{align*}
\sigma\textup{maex}(n) := \sum_{\pi \in \mathcal{P}(n)} \textup{maex}(\pi).
\end{align*}
He \cite[Theorem 1.1]{chernmaxexcludnt} obtained the following generating function identity for $\sigma\textup{maex}(n)$: 
\begin{align}\label{gen_sigmamaex}
\sum_{n=1}^\infty \sigma\textup{maex}(n) q^n = \sum_{n=1}^\infty \frac{n}{(q;q)_{n-1}} \sum_{m=1}^\infty q^{m(n+1)} (-q;q)_{m-1}. 
\end{align}
Chern \cite[Theorem 1.3]{chernmaxexcludnt} also established an asymptotic formula  for $\sigma\textup{maex}(n)$.  More precisely,  he showed that 
\begin{align*}
\sigma\textup{maex}(n) \sim \sigma L(n), \quad \textrm{as} \,\,  n  \rightarrow  \infty,
\end{align*}
where
$\sigma L(n)$ denotes the sum of  the largest parts of all partitions of $n$.  Kessler and Livingston showed that the following Hardy-Ramanujan type asymptotic formula holds for $\sigma L(n)$: 
\begin{align*}
\sigma L(n) \sim  \frac{\log \frac{6n}{\pi^2} + 2 \gamma }{4 \pi \sqrt{2n}} \exp\left( \pi \sqrt{\frac{2n}{3}} \right), \quad \textrm{as} \,\,  n  \rightarrow  \infty.
\end{align*}
Chern \cite[Theorem 2.2]{chernmaxexcludnt} also derived a formula for $\sigma\textup{maex}(n)$,  that connects the divisor function and the coefficients of Ramanujan's  $q$-series $\sigma(q)$,  defined in \eqref{sigma(q)} below.
%He connects the generating function of $\sigma \textup{maex}(n)$ to
%a well-known (??) mock theta function (and some algebraic number theory??)
%({\bf State asymptotic results of Chern} )

In \cite{andrewsnewmanI}, Andrews and Newman studied another arithmetic function, namely, 
\begin{align*}
a(n)= \sum_{ \pi \in \mathcal{P}(n) \atop \textup{mex}(\pi) \, \textup{odd}} 1.
\end{align*} 
They\footnote{Note that there is a typo in  \cite[p.~250, Theorem 1.2]{andrewsnewmanI},  in which the words odd and even are exchanged.} observed that $ \sigma\textup{mex}(n)\equiv a(n) \pmod 2$ and that $a(n)$ is almost always even and is odd exactly when $n$ is of the form $j(3j\pm 1)$. 

We study the questions raised by Andrews and Newman for the function $\sigma \textup{mex}(n)$ in \cite{andrewsnewmanI},
but restricted to partitions into distinct parts. We define the function $\sigma_d\textup{mex}(n)$ by
\begin{equation}\label{sigma_d defn}
\sigma_d \textup{mex}(n) := \sum_{\pi \in \mathcal{D}(n)} \textup{mex}(\pi),
\end{equation}
where $\mathcal{D}(n)$ denotes the collection of partitions of $n$ into distinct parts. We also define
\begin{align*}
a_d(n):=   \sum_{ \pi \in \mathcal{D}(n) \atop \textup{mex}(\pi) \, \textup{odd}} 1.
\end{align*}
In fact, the generating function for $a_d(n)$ was considered by Uncu \cite{uncuspin} in a different combinatorial context. 

In the next section we state the main results. 

\section{Main Results}
%First define $\sigma(q)$. 
Before stating the main results, let us begin with one of the most important $q$-series of Ramanujan, which has been a constant source of study from the point of view of both algebraic and analytic number theory. It is given by 
\begin{align}\label{sigma(q)}
\sigma(q) := \sum_{n=0}^{\infty} \frac{q^{\frac{n(n+1)}{2}}}{(-q; q)_n }.
\end{align}

Readers are encouraged to see \cite{andrewslostV1986}, \cite{andrewsmonthly86} and \cite{adh},  for deeper results related to $\sigma(q)$.  Surprisingly,  the generating function for $\sigma_d\textup{mex}(n)$ is directly connected to $\sigma(q)$ as stated below.
\begin{theorem}\label{sigma_d_mex}
We have
\begin{align*}
\sum_{n=0}^{\infty}\sigma_d \textup{mex}(n)q^n=(-q;q)_{\infty}\sigma(q).
\end{align*}
\end{theorem}

The next result gives us an asymptotic formula for $\sigma_d\textup{mex}(n)$.
\begin{theorem}\label{asym_sigma_d_mex}
We have
\begin{align} \label{asym_sigma_d_mex_equation}
\sigma_d{\rm mex}(n)   \sim   \frac{  \exp\left( \pi \sqrt{\frac{n}{3}}   \right) }{2 (3n^3)^{1/4} }, \quad  \textrm{as}\,\,  n \rightarrow \infty.
\end{align}
%where 
%\begin{align}\label{g(n)}
%g(n)= \sum_{i=0}^{ \left[  \frac{\sqrt{8n+1}- 1 }{2} \right] } \frac{1}{2^i} \left(  1 + \frac{3 i(i+1)}{8n} \right)  \exp \left( - \frac{\pi}{4 \sqrt{3}}  \frac{i(i+1)}{\sqrt{n}} \right).
%\end{align}
\end{theorem}

Now we will state a result due to Uncu \cite[Theorem 3]{uncuspin}.
\begin{theorem}
Let $U(n)$ be the sequence of numbers defined by
\begin{align*}
\sum_{n=0}^{\infty} U(n) q^n = (-q; q)_{\infty} \sum_{n = 0}^{\infty} \frac{(-1)^n q^{\binom {n+1}{2}}}{(-q; q)_n}.
\end{align*}
Then $U(n) \geq 0$ for all $n \geq 0$.
\end{theorem}
Interestingly, we observe that the generating function for $U(n)$ and $a_d(n)$ are indeed the same.

\begin{theorem}\label{a_d(n)}
We have
\begin{equation}\label{a_dfinalgfn}
\sum_{n=0}^{\infty} U(n) q^n =  \sum_{n = 0}^{\infty}  a_d(n) q^n = (-q; q)_{\infty} \sum_{n = 0}^{\infty} \frac{(-1)^n q^{\binom {n+1}{2}}}{(-q; q)_n}.
\end{equation}
\end{theorem}
An immediate consequence of this result is
\begin{corollary}\label{gen_uncu}
For any $n \geq 0$,  we have 
\begin{align*}
  U(n) > 0 \quad \textrm{except for}\,\, n=1.
\end{align*}
\end{corollary}

Andrews and Newman also defined $\textup{moex}(\pi)$ to be the smallest odd integer missing from $\pi$. This naturally led them to define the function $\sigma \textup{moex}(n)$, given by
\begin{equation*}
 \sigma \textup{moex}(n):= \sum_{\pi \in \mathcal{P}(n)} \textup{moex}(\pi).
\end{equation*}

We analogously define a quantity for partitions into distinct parts and study its generating function. Define
\begin{equation*}
\sigma_d \textup{moex}(n):= \sum_{\pi \in \mathcal{D}(n)} \textup{moex}(\pi).    
\end{equation*}

\begin{theorem}\label{Sigma_d_MOEX}
The generating function for $\sigma_d \textup{moex}(n)$ is
  \begin{align*}
   \sum_{n = 0}^{\infty} \sigma_d \textup{moex}(n)q^n &= (-q; q)_{\infty}\left(1+ 2\sum_{n = 1}^{\infty} \frac{q^{n^2}}{(-q; q^2)_{n}} \right) \\  
   &= (-q; q)_{\infty}\left(1 + 2\sum_{n=1}^{\infty} (-1)^{n-1} q^n (q^2; q^2)_{n-1}\right).
  \end{align*}
In other words,  
\begin{equation*}
 \sum_{n = 0}^{\infty} \sigma_d \textup{moex}(n)q^n = (-q; q)_{\infty}\left(1 + \sigma^*(-q)\right),
\end{equation*}
where \begin{equation*} \sigma^*(q):=2 \sum_{n = 1}^{\infty} \frac{(-1)^n q^{n^2}}{(q; q^2)_{n}}.
\end{equation*}
%(Write a few words connecting $\sigma(q)$ and $\sigma^{*}(q)$)
\end{theorem}
In literature,  the $q$-series 
 $\sigma(q)$ and $\sigma^{*}(q)$ are found to appear simultaneously at many
places. Andrews,  Dyson and Hickerson \cite{adh} proved that the coefficients of these two $q$-series are very small and related to the arithmetic of the quadratic real field $\mathbb{Q}(\sqrt{6})$. 
%Again,  we define
%\begin{equation}
%\sigma_d \textup{meex}(n):= \sum_{\pi \in \mathcal{D}(n)} \textup{meex}(\pi),
%\end{equation}
%
%\begin{theorem}\label{GFN_Sigma_d_MEEX}
%
% We then have the following generating function for $\sigma_d \textup{meex}(n)$.
%\begin{align}
%  \sum_{n \geq 0} \sigma_d \textup{meex}(n)q^n=2(-q; q)_{\infty}\sigma(q^2). \label{gfn_sigma_dmeex}
%  \end{align}
%
%\end{theorem}

Next,  similar to the definition \eqref{sigma_d defn},  we define $\sigma_d\textrm{maex}(n)$ as 
\begin{align*}
\sigma_d\textrm{maex}(n):= \sum_{\pi \in \mathcal{D}(n)} \textrm{maex}(n).
\end{align*}
The next result provides us a generating function identity for $\sigma_d\textrm{maex}(n)$.
\begin{theorem}\label{MAEX_D(N)}
We have
\begin{align*}
\sum_{n=0}^{\infty}\sigma_d \textup{maex}(n)q^n= \sum_{k=1}^{\infty}k(-q;q)_{k-1}\sum_{m=1}^{\infty}q^{\frac{m(m+1)}{2}+km}.
\end{align*}
\end{theorem}
One can easily observe that Theorem \ref{MAEX_D(N)} is an analogue of Chern's identity \cite[(2)]{chernmaxexcludnt}. Before going on to the proofs of our main results, we establish an auxiliary result in the next section.  

\section{Preliminaries}
For the proof of the asymptotic formula for $\sigma_d \textup{mex}(n)$, namely, Theorem \ref{asym_sigma_d_mex}, we derive a monotonic property of $\sigma_d \textup{mex}(n)$, which is itself of independent interest. 

\begin{proposition}\label{monotonic sigma d}
For $n \geq 7$, we have that 
$$
\sigma_d \textup{mex}(n+1) > \sigma_d \textup{mex}(n).
$$
\end{proposition} 
\begin{proof}
We consider two cases depending on whether $n$ is a triangular number or not.

\textbf{Case 1:} Suppose $n \geq 7$ is not a triangular number, that is, $n \neq 1+2+ \cdots + k$ for any positive integer $k$. 
We define a map $f$ from $\mathcal{D}(n)$ to $\mathcal{D}(n+1)$ as follows:
Consider a distinct parts partition $\pi$ of $n$. It will be of the form
\begin{equation}\label{pi in Dn}
\pi = (a_1, a_2, \dots, a_k) \quad \text{with} \ a_1 > a_2 > \cdots > a_k \geq 1, \ a_1 + a_2 + \cdots + a_k = n.
% & \Longleftrightarrow n+1 = (a_1 + 1) + \cdots + a_k. \label{D n+1}
\end{equation} 
Then let $f(\pi) := \pi^{'} = (a_1 + 1, a_2, \dots, a_k)$. Clearly, this is a distinct parts partition of $n+1$. Also see that $(a_1 + 1) - a_2 \geq 2$, i.e., a partition of the form $f(\pi)$ has a gap of at least \textit{two} between its largest part and the next largest part. Moreover, $f$ preserves minimal excludants, that is $\textup{mex}(\pi) = \textup{mex}(\pi^{'})$. This is because as $n$ is not a triangular number, $\pi = (a_1, \dots, a_k)$ cannot be the partition $(k, k-1, \dots, 1)$ and consequently, there exists an $i, \ 1 \leq i \leq k$ such that $a_i - a_{i+1} > 1$ with the interpretation that $a_{k+1} = 0$. Choosing $i_0$ to be the largest integer $i$ for which $a_i - a_{i+1} > 1$ implies that $\textup{mex}(\pi) = \textup{mex}(\pi^{'}) = a_{i_0 + 1} + 1$. 

For example, if $n = 9$, then the partition $\pi_1 = (4, 3, 2)$ satisfies $k=i_0=3$ with $\textup{mex}(\pi_1) = \textup{mex}(f(\pi_1)) = 1$ and the partition $\pi_2 = (8, 1)$ satisfies $k=2, i_0 = 1$ with $\textup{mex}(\pi_2) = \textup{mex}(f(\pi_2)) = 2$. 
Thus, we have matched up the minimal excludants of partitions in $\mathcal{D}(n)$ with those of a certain collection of partitions in $\mathcal{D}(n+1)$ and hence $\sigma_d \textup{mex}(n) \leq \sigma_d \textup{mex}(n+1)$. We next exhibit a partition in $\mathcal{D}(n+1)$ that does not lie in the image of $f$, thus proving that the inequality  $\sigma_d \textup{mex}(n+1) > \sigma_d \textup{mex}(n)$ holds. The idea is to get hold of a partition $\pi$ in which both $\ell(\pi)$ and $\ell(\pi) - 1$ occur as parts so that it cannot be in the image of $f$.

\textbf{If $n$ is odd:} Consider the partition $\lambda_1 = \left(\frac{n+1}{2}, \frac{n-1}{2}, 1\right)$ of $n+1$. Note that $\frac{n+1}{2} - \frac{n-1}{2} = 1$ and that $\lambda_1$ has distinct parts as $(n-1)/2 > 1$.

\textbf{If $n$ is even:} Look at the partition $\lambda_2 = \left(\frac{n}{2} + 1, \frac{n}{2}\right)$ of $n+1$, which is again in $\mathcal{D}(n+1) \setminus f(\mathcal{D}(n))$. \\

\textbf{Case 2:} Assume that $n \geq 7$ is a triangular number, so that $n = k+ (k-1) + \cdots + 1$ for a unique positive integer $k$. As in Case $1$, for any $\pi$ in $\mathcal{D}(n)$ other than the partition $\mu = (k, k-1, \dots, 1)$, the partition $f(\pi)$ lying in $\mathcal{D}(n+1)$ will have the same minimal excludant as $\pi$. But for $\mu$, whose minimal excludant is $k+1$, we see that $\textup{mex}(f(\mu)) = k$. This means that $\sum_{\pi \in \mathcal{D}(n)} \textup{mex}(\pi) = 1 + \sum_{\pi \in f(\mathcal{D}(n))} \textup{mex}(\pi)$. In other words, to prove the required inequality we have to show that the contribution of the partitions from $\mathcal{D}(n+1) \setminus f(\mathcal{D}(n))$ to the sum of minimal excludants is at least $2$. We once again proceed according to the parity of $n$.

\textbf{If $n$ is odd:} Consider the partition $\nu_1 = \left(\frac{n+1}{2}, \frac{n-1}{2}, 1\right)$ of $n+1$. Observe that $(n-1)/2 > 2$ and so $\nu_1 \in \mathcal{D}(n+1) \setminus f(\mathcal{D}(n))$ with its minimal excludant being $2$.

\textbf{If $n$ is even:} In this case, we look at two partitions given by $\nu_2 = (\frac{n}{2} + 1, \frac{n}{2})$ and $\nu_3 = (\frac{n}{2}, \frac{n}{2} - 1, 2)$. Clearly, $\nu_2 \in \mathcal{D}(n+1) \setminus f(\mathcal{D}(n))$ with minimal excludant $1$. Again, note that $\nu_3$ has distinct parts since $\frac{n}{2} - 1 > 2$ and hence $\textup{mex}(\nu_3) = 1$.

With this, we complete the proof of the proposition.
\end{proof}
In the upcoming section, we provide the proofs of our results.

\section{Proofs of the main results}

\subsection{Proof of Theorem \rm{\ref{sigma_d_mex}}}

\subsubsection{First proof of Theorem \rm{\ref{sigma_d_mex}}}
Let $p_d^{mex}(m, n)$ be the number of partitions of $n$ into distinct parts whose minimal excludant is $m$.
Then, we have 
\begin{align}
 \sum_{n=0}^{\infty}\sum_{m=1}^{\infty} p_d^{mex}(m, n)z^m q^n &= \sum_{m=1}^{\infty} z^m q^1 \cdot q^2 \cdots q^{m-1} \prod_{k=m+1}^{\infty}(1+ q^k)
\nonumber \\
 &=  \sum_{m=1}^{\infty} z^m q^{\binom m2}\prod_{k=m+1}^{\infty}(1+ q^k)
 \nonumber \\
 &= (-q; q)_{\infty} \sum_{m=1}^{\infty} \frac{z^m q^{\binom m2}}{(-q; q)_m}. \label{prediffwrtz_sigma_dmex}
\end{align}
Differentiating both sides of \eqref{prediffwrtz_sigma_dmex} with respect to $z$ and putting $z=1$, we get
\begin{equation*}
 \sum_{n=0}^{\infty}\left(\sum_{m=1}^{\infty} m p_d^{mex}(m, n)\right) q^n = (-q; q)_{\infty} \sum_{m=1}^{\infty} \frac{m q^{\binom m2}}{(-q; q)_m}.
\end{equation*}
But $ \displaystyle\sum_{m=1}^{\infty}m p_d^{mex}(m, n) = \sigma_d \textup{mex}(n),$ the sum of minimal excludants in all the partitions of $n$ into distinct parts.
Thus,
\begin{equation}\label{sigma_d mexgfn}
  \sum_{n=0}^{\infty} \sigma_d \textup{mex}(n)  q^n = (-q; q)_{\infty} \sum_{m=1}^{\infty} \frac{m q^{\binom m2}}{(-q; q)_m}.
\end{equation}

%We link up this generating function with Ramanujan's %function $\sigma(q)$ given by
%\begin{equation}\label{Rama_sigma}
%\sigma(q) := \sum_{n \geq 0} \frac{q^{n(n+1)/2}}{(-q; %q)_n}. 
%\end{equation}

We now show that the sum on the right hand side of \eqref{sigma_d mexgfn} is nothing but Ramanujan's series $\sigma(q)$. Start with
{\allowdisplaybreaks
 \begin{align*}
 \sum_{n=0}^{\infty} \frac{q^{n(n+1)/2}}{(-q; q)_n} &= \sum_{n=0}^{\infty}\frac{(n+1) - n}{(-q; q)_n} q^{n(n+1)/2} \\
 &= \sum_{n=0}^{\infty} \frac{(n+1) q^{n(n+1)/2}}{(-q; q)_n} - \sum_{n=0}^{\infty} \frac{n q^{n(n+1)/2}}{(-q; q)_n} \\
 &= \sum_{n=1}^{\infty} \frac{n q^{n(n-1)/2}}{(-q; q)_{n-1}} - \sum_{n=1}^{\infty} \frac{n q^{n(n+1)/2}}{(-q; q)_n} \\
 &= \sum_{n=1}^{\infty} \frac{n q^{n(n-1)/2}}{(-q; q)_{n-1}}\left( 1 - \frac{q^n}{1+q^n}\right)\\
 &= \sum_{n=1}^{\infty} \frac{n q^{n(n-1)/2}}{(-q; q)_{n}}.
\end{align*}}

Therefore, from \eqref{sigma_d mexgfn}, we deduce that
\begin{equation}\label{sigma_d_sigmaq}
 \sum_{n=0}^{\infty} \sigma_d \textup{mex}(n)  q^n = (-q; q)_{\infty} \sigma(q).
\end{equation}
\hfill{\qed}\\
%\end{proof}
We now give an alternate proof of \eqref{sigma_d_sigmaq}, on the lines of Andrews and Newman's second proof of the generating function for
$\sigma \textup{mex}(n)$. (see \cite[p. $251$]{andrewsnewmanI})

\subsubsection{Second proof of Theorem \ref{sigma_d_mex}}

Let $\mathcal{D}_i (n)$ denote the number of partitions of $n$ into distinct parts for which 
$\textup{mex}(\pi) > i$. Then we claim that 
\begin{equation}\label{claim}
\mathcal{D}_i (n) = p_d \left(n - \frac{i(i+1)}{2}, \ i  \right), 
\end{equation}
where $p_d (m, \ i)$ denotes the number of partitions of $m$ into distinct parts with smallest part 
greater than $i$. To see this, start with a distinct parts partition $\pi$ of $n$ with 
$\textup{mex}(\pi) > i$. By the definition of minimal excludant, the integers $1$ through $i$ must all occur as parts in $\pi$. Moreover, since $\pi$ is a distinct parts partition, each of the numbers
$1$ to $i$ appears exactly once in $\pi$. Subtract the quantity $1+2+\cdots + i$ from $\pi$. This
gives a distinct parts partition $\pi'$ of $n - (1 + 2+ \cdots + i)$ (since we began with a 
distinct parts partition $\pi$, removing some parts from it doesn't affect its distinct nature).
Now, since $\pi$ has only one copy of each of $1$ to $i$, $\pi'$ will not have any parts less than or equal to
$i$. Therefore $\pi'$ is a distinct parts partition of $n - \frac{i(i+1)}{2}$ with smallest part 
greater than $i$.

Conversely, starting with a distinct parts partition $\lambda$ of $n - \frac{i(i+1)}{2}$ with $s(\pi) > i$,
we add the quantity $1 + 2 + \cdots + i$ to $\lambda$ to get a distinct parts partition $\lambda'$ 
(since $\lambda$ had no parts less than or equal to $i$) with the integers $1$ to 
$i$ all occurring as parts. This means that $\textup{mex}(\lambda') > i$. Hence, this bijection
proves the claim in \eqref{claim}.

From the definition of $\mathcal{D}_i (n)$, $\sigma_d \textup{mex}(n)$ can be expressed as
\begin{equation}\label{sigma_d_D_i}
\sigma_d \textup{mex}(n) = \sum_{i = 0}^{\infty} \mathcal{D}_i (n), 
\end{equation}
since each distinct parts partition $\pi$ with $\textup{mex}(\pi) = i$ is counted $i$ times on the right hand side of \eqref{sigma_d_D_i}, 
once in each of $\mathcal{D}_0 (n), \mathcal{D}_1 (n), \dots, \mathcal{D}_{i-1} (n)$. On the left hand side of \eqref{sigma_d_D_i},
we add together the minimal excludants over all the distinct parts partitions, thus each distinct parts partition $\pi$ contributes a weight $\textup{mex}(\pi)$ to it. Thus, on both sides of equation 
\eqref{sigma_d_D_i}, each distinct parts partition contributes the same number and hence the identity holds.

Now, the generating function of distinct parts partitions with $s(\pi) > i$ is simply
\begin{equation*}
\sum_{n=0}^{\infty} p_d (n, \ i) q^n 
= (-q^{i+1}; q)_{\infty}.
\end{equation*}

Therefore, $\mathcal{D}_i (n)$, which is the number of distinct parts partitions of $n - 
\frac{i(i+1)}{2}$ with smallest part greater than $i$, will be the coefficient of 
$q^{n - \frac{i(i+1)}{2}}$ in $(-q^{i+1}; q)_{\infty}$. Equivalently, this is the coefficient
of $q^n$ in $q^{\frac{i(i+1)}{2}} (-q^{i+1}; q)_{\infty}$.

Thus,
\begin{equation*}
\sum_{n=0}^{\infty}\mathcal{D}_i (n) q^n = q^{\frac{i(i+1)}{2}} (-q^{i+1}; q)_{\infty} = 
(-q; q)_{\infty} \frac{q^{\frac{i(i+1)}{2}}}{(-q; q)_i}.
\end{equation*}
We are ready to obtain the generating function of $\sigma_d \textup{mex}(n)$. Starting with
\eqref{sigma_d_D_i}, we get
\begin{align*}
\sum_{n=0}^{\infty} \sigma_d \textup{mex}(n) q^n = \sum_{n=0}^{\infty} q^n 
\sum_{i=0}^{\infty} \mathcal{D}_i (n) &= \sum_{i=0}^{\infty}\sum_{n=0}^{\infty} 
\mathcal{D}_i (n) 
q^n \\
&= \sum_{i=0}^{\infty} (-q; q)_{\infty} \frac{q^{\frac{i(i+1)}{2}}}{(-q; q)_i} \\
&= (-q; q)_{\infty} \sum_{i=0}^{\infty} \frac{q^{\frac{i(i+1)}{2}}}{(-q; q)_i} = (-q; q)_{\infty}
\sigma(q).
\end{align*}
\hfill{\qed}

\subsection{Proof of Theorem {\rm \ref{asym_sigma_d_mex}}}
%\begin{proof}[ Theorem {\rm \ref{asym_sigma_d_mex}}][]

We first state an important asymptotic result for coefficients of a power series, due to Ingham \cite{Ingham}.

\begin{proposition}\label{Ingham}
Let $ A(q) = \sum_{n=0}^\infty a(n) q^n$ be a power series with radius of convergence $1$. Assume that $ \{ a(n) \}$ is a weakly increasing sequence of non-negative real numbers.  If there are constants $\alpha,  \beta \in \mathbb{R}$,  and $C >0$ such that 
\begin{align*}
A(e^{-t}) \sim \alpha \, t^{\beta} \exp\left( \frac{C}{t} \right),  \,\, \text{ as} \,\, t \rightarrow 0^{+},
\end{align*}
then we have
\begin{align}\label{Ingham_asymptotic}
a(n) \sim \frac{\alpha}{2 \sqrt{\pi}} \frac{C^{\frac{2\beta +1}{4} }}{n^{\frac{2\beta +3}{4}}}  \exp\left( 2 \sqrt{C n} \right),  \,\, \text{as} \,\,  n \rightarrow \infty.  
\end{align}
\end{proposition}
\noindent
\begin{proof}[Theorem \ref{asym_sigma_d_mex}][]
Recall that the generating function for $\sigma_d{\rm mex}(n)$ is given by Theorem \ref{sigma_d_mex},  namely, 
\begin{align*}
\sum_{n=0}^{\infty}\sigma_d \textup{mex}(n)q^n=(-q;q)_{\infty}\sigma(q) := B(q). 
\end{align*}
In a famous work on quantum modular forms,  Zagier \cite[p.~7]{Zagier} pointed out that,  for $t\rightarrow 0^{+}$,
\begin{align}\label{Zagier_asymptotic}
\sigma\left(e^{-t} \right) =  2 -2 t + 5 t^2 - \frac{55}{3} t^3 + \frac{1073}{12} t^4 - \frac{32671}{60} t^5 + \cdots.  
\end{align}
Now, using the transformation formula for the Dedekind's eta-function,  one can show that,  for $t\rightarrow 0^{+}$,
\begin{align}\label{dedekind}
\frac{1}{ \left(e^{-t}; e^{-t}  \right)_\infty} \sim \sqrt{\frac{t}{2\pi}} \exp\left( \frac{\pi^2}{6 t} \right). 
\end{align}
Euler's identity suggests that 
\begin{align}\label{Euler}
( -q; q)_\infty = \frac{1}{(q; q^2)_\infty} = \frac{ \left( q^2; q^2 \right)_\infty }{(q; q)_\infty}. 
\end{align}
Therefore,  in view of \eqref{dedekind} and \eqref{Euler},  we can derive,  for $t \rightarrow 0^{+}$,  
\begin{align}\label{asym_distinct}
\left( - e^{-t}; e^{-t} \right)_\infty \sim \frac{1}{\sqrt{2}} \exp\left(  \frac{\pi^2}{12 t } \right).  
\end{align}
Finally,  combining \eqref{Zagier_asymptotic} and \eqref{asym_distinct},  we arrive at
\begin{align*}
B\left( e^{-t} \right) \sim \sqrt{2} \exp\left(  \frac{\pi^2}{12 t } \right), \quad {\rm as} \,\,  t \rightarrow 0^{+}.
\end{align*}
Again, by Proposition \ref{monotonic sigma d} we know that the sequence $\{ \sigma_d {\rm mex}(n) \}$, for $n \geq 7$, is an increasing sequence of positive integers. Now we are ready to invoke Proposition \ref{Ingham},  with $\alpha =\sqrt{2},   \beta=0$ and $C=  \displaystyle\frac{\pi^2}{12}$.  Substituting these constants in \eqref{Ingham_asymptotic},  we obtain 
\begin{align*}
\sigma_d{\rm mex}(n) \sim  \frac{ 1}{2 (3 n^3)^{1/4}} \exp\left(  \pi \sqrt{\frac{n}{3}} \right),  \quad \text{as}\,\, n \rightarrow \infty.
\end{align*}
This finishes the proof.  
%\hfill{\qed} \\
\end{proof}

In the next subsection, we provide proofs of other results. 
\subsection{Proofs of other results}

%\subsection{Proof of Theorem \ref{a_d(n)}}
\begin{proof}[ Theorem {\rm \ref{a_d(n)}}][]
Recall that $a_d(n)$ counts the number of distinct parts partitions of $n$ with an odd minimal excludant. So the least integer missing from such a 
partition can only be of the form $2n+1$ for some $n \geq 0$. And all the integers from 1 through $2n$ should occur exactly once and for the 
integers greater than $2n+1$, they may occur at most once. Putting this together, we may write
\begin{equation}\label{pd_omex}
\sum_{n=0}^{\infty}\sum_{m=1}^{\infty} p_d^{omex}(m, n)z^m q^n = \sum_{k=0}^{\infty} z^{2k+1} q^1 \cdot q^2 \cdots q^{2k} \prod_{\ell=2k+2}^{\infty}(1+ q^{\ell}),    
\end{equation}
where $p_d^{omex}(m, n)$ denotes the number of distinct parts partitions of $n$ with an odd minimal excludant $m$. Putting $z=1$ in \eqref{pd_omex}, we get 
\begin{equation}\label{a_dn_gfn}
\sum_{n=0}^{\infty} a_d(n) q^n = \sum_{k=0}^{\infty} q^{\binom {2k+1}{2}} \prod_{\ell=2k+2}^{\infty}(1 + q^{\ell})= (-q;q)_{\infty} \sum_{n= 0}^{\infty}\frac{q^{\binom {2n+1}{2}}}{(-q; q)_{2n+1}}.   
\end{equation}
Consider the rightmost sum in \eqref{a_dn_gfn}. Rewriting it, we get
{\allowdisplaybreaks \begin{align*}
\sum_{n=0}^{\infty}\frac{q^{\binom {2n+1}{2}}}{(-q; q)_{2n+1}} &= \sum_{n=0}^{\infty} \frac{q^{1+ \cdots + 2n}}{(1+q)\cdots(1+q^{2n+1})}\nonumber \\
&= \sum_{n=0}^{\infty} \frac{q^{1+ \cdots + 2n}}{(1+q)\cdots(1+q^{2n})}\left\{ 1 - \frac{q^{2n+1}}{1+q^{2n+1}}\right\} \nonumber\\
&= \sum_{n=0}^{\infty} \frac{q^{1+ \cdots + 2n}}{(1+q)\cdots(1+q^{2n})} - \sum_{n=0}^{\infty} \frac{q^{1+ \cdots + (2n+1)}}{(1+q)\cdots(1+q^{2n+1})} \nonumber\\
&= \sum_{n=0}^{\infty} \frac{q^{\binom {2n+1}{2}}}{(-q; q)_{2n}} - \sum_{n=0}^{\infty} \frac{q^{\binom {2n+2}{2}}}{(-q; q)_{2n+1}}\nonumber\\
&= \sum_{n=0}^{\infty} \frac{(-1)^n q^{\binom {n+1}{2}}}{(-q; q)_n}.
\end{align*} }
Putting this in \eqref{a_dn_gfn}, we see that
\begin{equation*}
 \sum_{n = 0}^{\infty} a_d(n) q^n = (-q; q)_{\infty} \sum_{n = 0}^{\infty} \frac{(-1)^n q^{\binom {n+1}{2}}}{(-q; q)_n}.
\end{equation*}
This completes the proof.
\end{proof}

\begin{proof}[Corollary  {\rm \ref{gen_uncu}}][]
We note that $a_d(n)$ is non-negative for all $n \geq 0$, since it counts certain kind of partitions. Moreover, for
$n>1$,  we always have a partition of $n$ into distinct parts with an odd minimal excludant, namely,  the partition $n$,  
where the minimal excludant is $1$.  So
$a_d(n) > 0$ for all $n > 1$,  and hence by \eqref{a_dfinalgfn}, we conclude that  $U(n)>0$ for all $n>1$.
\end{proof}
%\subsection{Proof of Theorem \ref{Sigma_d_MOEX}} 
Uncu \cite[Theorem 3.2]{uncuspin}, remarks that the infinite series in \eqref{a_dfinalgfn} is a false theta function
studied by Rogers. Further, in the same paper, he gives a combinatorial explanation of the fact that the coefficients on the right hand side of
\eqref{a_dfinalgfn} are non-negative. But, by Corollary \ref{gen_uncu}, via our interpretation in terms of minimal excludant, 
we have shown that all but one of the coefficients, namely $a_d(1)$, are infact positive.

\begin{proof}[Theorem {\rm {\ref{Sigma_d_MOEX}}}][]
Let  $p_d^{moex}(m, n)$ denote the number of distinct parts partitions $\pi$ of $n$ with $\textup{moex}(\pi)=m$. Consider the following double sum
{\allowdisplaybreaks \begin{align*}
 \sum_{n=0}^{\infty} \sum_{m=1}^{\infty} p_d^{moex}(m, n)z^m q^n &= (-q^2; q^2)_{\infty}\sum_{k=0}^{\infty} z^{2k+1}q^{1+3+\cdots+(2k-1)} (-q^{2k+3}; q^2)_{\infty} \nonumber\\
&= (-q^2; q^2)_{\infty}\sum_{k=0}^{\infty} z^{2k+1} q^{k^2}(-q^{2k+3}; q^2)_{\infty} \nonumber \\
&= (-q; q^2)_{\infty} (-q^2; q^2)_{\infty} \sum_{k=0}^{\infty} \frac{z^{2k+1} q^{k^2}(-q^{2k+3}; q^2)_{\infty}}{(-q; q^2)_{\infty}} \nonumber \\
&=  (-q; q)_{\infty} \sum_{k=0}^{\infty}  \frac{z^{2k+1} q^{k^2}}{(-q; q^2)_{k+1}}.
 \end{align*} }
 Differentiate with respect to $z$ and then put $z=1$ to get
{\allowdisplaybreaks \begin{align*}\label{gfn_sigma_dmoex}
  \sum_{n=0}^{\infty} \sigma_d \textup{moex}(n)q^n &= (-q; q)_{\infty} \sum_{n=0}^{\infty} \frac{(2n+1) q^{n^2}}{(-q; q^2)_{n+1}}\\
  &= (-q; q)_{\infty} \sum_{n=0}^{\infty} \frac{(2n+1) q^{n^2}}{(-q; q^2)_{n}}\left\{ 1-\frac{q^{2n+1}}{1+q^{2n+1}}\right\} \\
 &= (-q; q)_{\infty} \sum_{n=0}^{\infty} \frac{(2n+1) q^{n^2}}{(-q; q^2)_{n}} - (-q; q)_{\infty} \sum_{n=0}^{\infty}\frac{(2n+1) q^{(n+1)^2}}{(-q; q^2)_{n+1}}\\
 &= (-q; q)_{\infty} \sum_{n=0}^{\infty} \frac{(2n+1) q^{n^2}}{(-q; q^2)_{n}} - (-q; q)_{\infty} \sum_{n=1}^{\infty} \frac{(2n-1) q^{n^2}}{(-q; q^2)_{n}}\\
 &= (-q; q)_{\infty} + 2(-q; q)_{\infty}\sum_{n=1}^{\infty} \frac{q^{n^2}}{(-q; q^2)_{n}}.
  \end{align*} }
  Alternatively, using Chern's result \cite[Proposition 2.1]{chernmaxexcludnt} with $x=1, y=-1$ gives us
  \begin{equation*}
   \sum_{n=0}^{\infty} \sigma_d \textup{moex}(n)q^n = (-q; q)_{\infty} + 2(-q; q)_{\infty}\sum_{n=1}^{\infty} (-1)^{n-1} q^n (q^2; q^2)_{n-1}.
  \end{equation*}
\end{proof}
%\subsection{Proof of the Corollary \ref{a_dfinalgfn}}
%\subsection{Proof of {\rm Theorem \ref{MAEX_D(N)}}

\begin{proof}[Theorem {\rm \ref{MAEX_D(N)}}][]
Suppose $\pi_d$ is a distinct parts partition of $n$ with maximal excludant $k$. 
Note that $k \geq 1$, since partitions with maximal excludant $0$ do not contribute to the sum $\sum_{\pi \in \mathcal{D}(n)} \textup{maex}(\pi) = 
\sigma_d \textup{maex}(n)$. 
We can divide $\pi_d$ into two components, ${\pi_d}'$ and  ${\pi_d}''$:\\
The first  component ${\pi_d}'$ is a distinct parts partition with parts $\leq k-1$; and the second one  ${\pi_d}''$ is a gapfree distinct parts partition with $s(\pi)=k+1$,  i.e.,  each integer between $s(\pi)$ and $\ell(\pi)$ also occurs as a part.
Observe that the second component ${\pi_d}''$ upon conjugation gives a gapfree partition in which the smallest part $s(\pi)=1$ and the 
largest part $\ell(\pi)$ appears exactly $k+1$ times and all other parts appear exactly once.
We consider a two variable generating function $D(z,q)$ for $p_{d, k}(n)$, the number of distinct parts partitions of $n$ with
maximal excludant $k$. In $D(z, q)$, the exponent of $z$ indicates the maximal excludant of a partition $\pi_d$ into
distinct parts, and the exponent of $q$, as always, keeps track of the number being partitioned by $\pi_d$.
\begin{align*}
D(z,q):= \sum_{n=0}^{\infty} \sum_{k=1}^{\infty} p_{d, k}(n) z^k q^n = \sum_{k=1}^{\infty}(-q;q)_{k-1}z^k\sum_{m=1}^{\infty}q^{1+2+...+(m-1)+(k+1)m}.
\end{align*}
Now differentiating $D(z,q)$  with respect to $z$ and substituting $z=1$, we get the generating function for $\sigma_d \textup{maex}(n)$,
\begin{align*}
\sum_{n=0}^{\infty}\sigma_d \textup{maex}(n)q^n=\sum_{k=1}^{\infty}k(-q;q)_{k-1}\sum_{m=1}^{\infty}q^{\frac{m(m+1)}{2}+km}.
\end{align*}
\end{proof}

\section{Concluding Remarks}
%It is desirable from the readers to try the bijective proof of $$\Delta_1(n) =  \sigma\textup{mex}(n) - \sigma\textup{mex}(n-1).$$ Similarly for  $\Delta_2(n)$ and  $\Delta_3(n)$.
Inspired by the work of Andrews and Newman,  in the current paper, we studied the minimal excludant over partitions into distinct parts. 
We have proved that the generating function for $\sigma_d{\rm mex}(n)$ is the product of the generating function for distinct parts partition function 
and Ramanujan's well-known $q$-series $\sigma(q)$. We also established a Hardy-Ramanujan type asymptotic  formula for  $\sigma_d{\rm mex}(n)$. 
%We feel that the function $g(n)$ could be simplified, but we are unable to simplify further. 
It would be interesting to find a Hardy-Ramanujan-Rademacher type exact formula for $\sigma_d{\rm mex}(n)$, analogous to the result \eqref{HRR type} of Grabner and Knopfmacher for $\sigma{\rm mex}(n)$. 

We also examined $a_d(n)$, 
which counts the number of distinct parts partitions with an odd minimal excludant.  Quite surprisingly,  we have observed that the generating function for $a_d(n)$ has been studied by Uncu in a different context,  which immediately improved Uncu's result \cite[Theorem 3]{uncuspin}.  Subsequently,  we studied $\sigma_d{\rm moex}(n)$ and 
its generating function has been expressed as the product of the generating function for the distinct parts partition function and $1+\sigma^{*}(-q)$. It is interesting that the function $\sigma^{*}(q)$
has mostly been seen to appear in the vicinity of $\sigma(q)$.
%Readers are encouraged to see recent interesting works \cite{barman-singh1,  barman-singh2,  Chakraborty-Roy,  sellers-dasilva} on Mex-related partition functions of Andrews and Newman.  
Recently,  using the theory of modular forms,  Barman and Singh \cite{barman-singh1,  barman-singh2},  and Chakraborty and Ray \cite{Chakraborty-Ray} found interesting congruence properties and density results for Mex-related partition functions.  Readers are encouraged to see the paper of da Silva and Sellers \cite{sellers-dasilva} for parity results and congruence properties related to Mex-related partition functions of Andrews and Newman. 

\textbf{Acknowledgements.}
We sincerely thank the anonymous referee for providing many useful suggestions,  particularly,  for outlining ideas for the proof of Theorem 2.2. The first author wants to thank Prof.  Meenakshi Rana for her continuous support. The second author wants to thank the Department of Mathematics,  Pt. Chiranji Lal Sharma Government College,  Karnal for a conducive research environment. The third author is a SERB National Post Doctoral Fellow (NPDF) supported by the fellowship PDF/2021/001090 and would like to thank SERB for the same. In addition, he wishes to thank IIT Indore for its conducive research environs and also his former employer IISER Berhampur for providing flexibility in working conditions. The last author wishes to thank SERB for the Start-Up Research Grant SRG/2020/000144.

\end{document}